\documentclass[a4paper,11pt]{amsart}
\usepackage[colorlinks]{hyperref}
\hypersetup{
    colorlinks=true,       
    linkcolor=blue,          
    citecolor=cyan,        
}
\usepackage{fullpage}
\usepackage{amsmath}
\usepackage{amsthm}
\usepackage{amssymb}
\usepackage{verbatim}

\def\rb{\mathbb{R}}
\def\nb{\mathbb{N}}
\def\zb{\mathbb{Z}}

\numberwithin{equation}{section}
\newtheorem{theorem}{Theorem}[section]
\newtheorem{lemma}[theorem]{Lemma}
\newtheorem{corollary}[theorem]{Corollary}

\newtheorem{definition}{Definition}[section]
\newtheorem{example}[theorem]{Example}

\numberwithin{equation}{section}

\begin{document}

\title{The explicit asymptotic formula of divisor function on average over values of quadratic polynomial}

\author{Nianhong Zhou}
\address{Department of Mathematics\\
East China Normal University\\
500 Dongchuan Road, Shanghai 200241, PR China}
\email{nianhongzhou@outlook.com}

\begin{abstract}
Let $F({\bf x})={\bf x}^tQ{\bf x}+\mathbf{b}^t{\bf x}+c\in\mathbb{Z}[{\bf x}]$ be a quadratic polynomial in $\ell (\ge 3 )$ variables ${\bf x} =(x_{1},...,x_{\ell})$, where $F({\bf x})$ is positive when ${\bf x}\in\mathbb{R}_{\ge 1}^{\ell}$, $Q\in {\rm M}_{\ell}(\mathbb{Z})$ is an $\ell\times\ell$ matrix  and its discriminant $\det\left(Q^t+Q\right)\neq 0$.
It gives an explicit asymptotic formula for the following sum
\[
\sum_{{\bf x}\in [1,X]^{\ell}\cap\mathbb{Z}^{\ell}}\tau\left(F({\bf x})\right),
\]
where $\tau$ is the divisor function.
\end{abstract}

\maketitle


\section{Introduction}\label{sec1}
Let $F({\bf x})$ be a quadratic polynomial with $\ell(\ge3)$ variables $x_1, x_2,.., x_{\ell}$ and integer coefficients. Unless stated explicitly otherwise, we shall write ${\bf x}$ for the vector
$(x_1, x_2, . . . , x_{\ell})^{t}\in\mathbb{Z}^{\ell}$.
We assume quadratic polynomial $F({\bf x})$ satisfies
\begin{equation}\label{fdef1}
F({\bf x})={\bf x}^tQ{\bf x}+\mathbf{b}^t{\bf x}+c,
\end{equation}
where $Q\in {\rm M}_{\ell}(\mathbb{Z})$ is an $\ell\times\ell$ matrix with entries $a_{ij}$, vector $\mathbf{b}=(b_1,..., b_{\ell})^t\in\mathbb{Z}^{\ell}$, $c\in\zb $ and suppose those coefficients satisfies the following conditions
\begin{align}\label{fdef2}
\begin{cases}
\min_{{\bf x}\in [1,X]^{\ell}}F({\bf x})>0\\
\Delta_{F}=\det\left(Q^t+Q\right)\neq 0.
\end{cases}
\end{align}
We have already shown the following theorem in \cite{Zhou2017}.
\begin{theorem}\label{tm1} Let $F$ be defined as above, $k\ge 2$ and $\ell\ge 3$. Then, for any $\varepsilon>0$ there exist constants $H_{k,0}(F)$, $H_{k,1}(F)$,..., and $H_{k,k-1}(F)\in \rb$, such that
\[\sum_{{\bf x}\in [1,X]^{\ell}\cap \zb^{\ell}}\tau_k(F({\bf x}))=\sum_{r=0}^{k-1}H_{k,r}(F)\int_{[1,X]^{\ell}}(\log F({\bf t}))^r{\rm d}{\bf t}+O_{k,F,\varepsilon}\left(X^{\ell-\frac{\ell-2}{\ell+2}\min\left(1,\frac{4}{k+1}\right)+\varepsilon}\right),\]
\[H_{k,r}(F)=\frac{1}{r!}\sum_{t=0}^{k-r-1}\frac{1}{t!}\left({\frac{{\rm d}^tL(s;k,F)}{{\rm d}s^t}}\bigg|_{s=1}\right) {\rm Res}\left((s-1)^{r+t}\zeta(s)^k; s=1\right),\]
where
\[L(s; k,F)=\prod_{p}\left(\sum_{m\in\nb}p^{-(\ell-1)m}S_F(p^m)F_k(p^m,s)\right)\]
with
\[S_F(q)=q^{-1}\sum_{a\in\zb_q^*}\sum_{{\bf h}\in(\zb_q)^{\ell}}e\left(\frac{a}{q}F({\bf h})\right),\]
$F_k(q,s)$ is multiplicative for $q\in\nb_{\ge 1}$ and
\[F_k(p^m,s)=p^{-ms}\left(\sum_{v=1}^{k-1}(1-p^{-s})^{v-1}\tau_{v}(p^{m-1})+(1-p^{-s})^{k-1}\tau_k(p^{m-1})\frac{p^s-1}{p-1}\right).\]
for any prime $p$ and integer $m\ge 1$.
\end{theorem}
Our focus is upon the case of the explicit asymptotic formula of $k=2$.
To giving our main results, we need some definitions.
\begin{definition}\label{def12}
The first discriminant of quadratic polynomial $F({\bf x})$ as above defined as
\begin{align}
\Delta_{F}=\det\left(Q^t+Q\right),
\end{align}
and the second discriminant of the quadratic polynomial $F({\bf x})$ is defined by
\begin{align}
\mathcal{H}_{F}=\begin{cases}\qquad\qquad\quad 0 & if~ \ell~ is~ odd\\
(-1)^{\ell/2}\det\left(Q^t+Q\right)& if~ \ell~ is~ even,
\end{cases}
\end{align}
where $F({\bf x})$ is defined by (\ref{fdef1})  and (\ref{fdef2}).
\end{definition}
\begin{definition}\label{def13}
The first $Q$-Root of quadratic polynomial $F({\bf x})$ is defined by
\begin{align}
\mathcal{R}_{F}=2c\Delta_{F}-{\bf b}^t\hat{Q}_S{\bf b}
\end{align}
and the second $Q$-Root of quadratic polynomial $F({\bf x})$ is defined by
\begin{align}
\mathcal{O}_{F}=\begin{cases}\qquad\qquad\qquad~ 0&  if~ \ell~ is~ even\\
(-1)^{(\ell+1)/2}\left(2\Delta_{F}c-{\bf b}^t\hat{Q}_S{\bf b}\right)&  if~ \ell~ is~ odd,
\end{cases}
\end{align}
where $\hat{Q}_S$ is the adjoint matrix of matrix $Q_S:=(Q^t+Q)$ and $F({\bf x})$ is defined by (\ref{fdef1})  and (\ref{fdef2}).
\end{definition}

Let $\left(\frac{\cdot}{p}\right)$ is Legendre quadratic residue symbol. Then with the above concepts, our main result as follows.
\begin{theorem}\label{thm1}
If $\Delta_{F}\neq 0$. Then for $X\ge 3$, any $\varepsilon>0$ and $\ell\ge 3$ we have
$$
\sum_{{\bf x}\in [1,X]^{\ell}\cap \zb^{\ell}}\tau(F({\bf x}))
=C_{F}(\tau)_mX^{\ell}\log X+C_{F}(\tau)_sX+O_{F,\varepsilon} \left(X^{\ell-\frac{\ell-2}{\ell+2}+\varepsilon}\right),
$$
where the constants are given by
$$
C_{F}(\tau)_m=2L(\ell, F),~
C_{F}(\tau)_s=\left(\gamma+\frac12I_{FL}+\frac{L'(\ell, F)}{L(\ell,F)}\right)C_{F}(\tau)_m,
$$
$\gamma$ is Euler constant,
\[
I_{FL}=\int_{[0, 1]^{\ell}}\log(F({\bf t}, X))\mathrm{d}{\bf t},~~~~
F({\bf t}, X)={\bf t}^tQ{\bf t}+\frac{{\bf b}^t}{X}{\bf t}+\frac{c}{X^2},
\]
\begin{align*}
L(s, F)
=\begin{cases}
f(s, F)\frac{\zeta(2s-\ell)}{\zeta(2s+1-\ell)}\prod\limits_{p>2}\left(1+\left(\frac{\mathcal{H}_{F}}{p}\right)
\frac{(1-p^{-1})p^{\frac{\ell}{2}-s}}{1-p^{\ell-2s-1}}\right)& if~\mathcal{R}_{F}=0\\
h(s, F)\prod\limits_{p\nmid 2\Delta_{F}\mathcal{R}_{F}}\left\{\left(1+\frac{\left(\frac{\mathcal{O}_{F}}{p}\right)}{p^{s-(\ell-1)/2}}\right)\left(1-\frac{\left(\frac{\mathcal{H}_{F}}{p}\right)}{p^{s-(\ell-2)/2}}\right)\right\}& if~\mathcal{R}_{F}\neq 0,
\end{cases}
\end{align*}
\[
f(s, F)=\prod_{p|2\Delta_{F}}\left\{\left(\frac{(1-p^{\ell-2s})(1-p^{\ell-s-1})}{1-p^{\ell-2s-1}}\right)\sum_{m\ge 0}\frac{\varrho_{F}(p^m)}{p^{ms}}\right\},
\]
\[
h(s, F)=\prod_{p|2\Delta_{F}\mathcal{R}_{F}}\left\{\left(1-p^{\ell-s-1}\right)\sum_{m\ge 0}\frac{\varrho_{F}(p^m)}{p^{ms}}\right\},
\]
and for any integer $n\ge 1$, where
\[\varrho_{F}(n):=\#\{{\bf x}\in(\zb_n)^{\ell}: F({\bf x})\equiv 0\bmod n\}.\]
\end{theorem}

\section{The proof of the Main Theorem}
\subsection{The treatment of $F_2(q,s)$} When $k=2$, by Theorem \ref{tm1} we have
\begin{align*}
F_2(p^m,s)=p^{-ms}\left(1+(1-p^{-1})\tau_2(p^{m-1})\frac{p-1}{p-1}\right).
\end{align*}
Thus
\[F_2(p^m,1)=p^{-m}\;\mbox{and}\; F_2^{\langle 1\rangle}(p^m,1)=-2m p^{-m}\log p.\]
$F_2(p^m,s)$ is multiplicative implies
\[F_2(q,1)=q^{-1} \;\mbox{and}\;  F_2^{\langle 1\rangle}(q,1)=-2q^{-1}\log q.\]
Hence
\[L(1; 2,F)=\sum_{q=1}^{\infty}\frac{S_F(q)}{q^{\ell}}\;\mbox{and}\; L^{\langle 1\rangle}(1; 2,F)=-2\sum_{q=1}^{\infty}\frac{S_F(q)\log q}{q^{\ell}}.\]
Therefore we define
\begin{equation}\label{lf2}
L(s,F)=\sum_{q=1}^{\infty}\frac{S_F(q)}{q^{s}},
\end{equation}
then
\[\sum_{{\bf x}\in [1,X]^{\ell}\cap \zb^{\ell}}\tau (F({\bf x}))=2H_{2,1}(F)X^{\ell}\log X+\left(H_{2,1}(F)I_{FL}+H_{2,0}(F)\right)X^{\ell}+O_{F}\left(X^{\ell-\frac{\ell-2}{\ell+2}+\varepsilon}\right),\]
where
\begin{equation}\label{ll0}
H_{2,1}(F)=L(\ell,F),\;\; H_{2,0}(F)=2\gamma L(\ell,F)+2 L^{'}(\ell,F)
\end{equation}
\subsection{The treatment of $S_F(q)$}
We consider the value of $S_F(p^m)$. Let $m\ge 1$ be an integer. It is easily seen that
\begin{equation}\label{21}
S_F(p^m)=\varrho_{F}(p^m)-p^{\ell-1}\varrho_{F}(p^{m-1}).
\end{equation}
Hence let us consider $\varrho_{F}(p^t)$. For every odd prime and integer $t\ge 1$,
\begin{align*}
\varrho_{F}(p^t)&=\#\{{\bf x}\in(\zb_{p^t})^{\ell}:{\bf x}^tQ{\bf x}+{\bf b}^t{\bf x}+c\equiv 0\bmod p^t\}\\
&=\#\{{\bf x}\in(\zb_{p^t})^{\ell}:{\bf x}^t{Q_S}{\bf x}+2{\bf b}^t{\bf x}+2c\equiv 0\bmod p^t\}.
\end{align*}
If $(\Delta_{F}, p)=1$, then do the following invertible linear transform
$${\bf s}={\bf x}-\Delta_{F}^{-1}\hat{Q}_S{\bf b},$$
where $\Delta_{F}^{-1}\Delta_F\equiv 1\bmod p^t$ and then
\begin{align}\label{phop}
\varrho_{F}(p^t)&=\#\{{\bf s}\in(\zb_{p^t})^{\ell}:{\bf s}^tQ_S{\bf s}\equiv \Delta_{F}^{-1}{\bf b}^t\hat{Q}_S{\bf b}-2c \bmod p^t\}\nonumber\\
&=\#\{{\bf s}\in(\zb_{p^t})^{\ell}:{\bf s}^t\left(\Delta_{F}Q_S\right){\bf s}+\mathcal{R}_{F}\equiv 0 \bmod p^t\},
\end{align}
where $\mathcal{R}_{F}$ is defined by Definition \ref{def13}. On the other hand, by the Proposition 12 of Chapter 12 in \cite{Coppel2006} we have the following lemma
\begin{lemma}Let $p$ be an odd prime. Let $Q_S$ is a symmetric $\ell\times\ell$ matrix in finite field $\mathbb{F}_{p}$. Also, let $\Delta_{F}$ defined by Definition \ref{def12}. Then, the quadratic form ${\bf s}^t\left (\Delta_{F}Q_S\right) {\bf s}$ equivalent to the following standard diagonal quadratic form
\begin{align*}
\mathcal{D}_{m}({\bf x})=x_1^2+x_2^2+. . . +\Delta_{F}^{\ell+1}x_{\ell}^2.
\end{align*}
Thus there exists an orthogonal matrix  $P_{F}$ in $\mathbb{F}_{p}$ such that
\begin{align*}
P_{F}^tQ_SP_{F}=\mathrm{diag}(1, 1, . . , \Delta_{F}^{\ell+1}).
\end{align*}
\end{lemma}
Using this lemma then easily seen that if $(2\Delta_{F}, p)=1$, one has
\begin{align}\label{ass8}
\varrho_{F}(p^t)=\#\{{\bf x}\in(\zb_{p^t})^{\ell} :x_1^2+x_2^2+. . . +\Delta_{F}^{\ell+1}x_{\ell}^2+\mathcal{R}_{F}\equiv0\bmod p^t\}.
\end{align}
Together with (\ref{phop}) and (\ref{ass8}), then
\begin{align*}
p^{t}S_F(p^t)&=\sum_{a\in\zb_{p^t}^*}\sum_{k_j\in\zb_{p^t}, j=1, 2, . , \ell}e\left(\frac{a}{p^t}(k_1^2+. . . +k_{\ell-1}^2+\Delta_{F}^{\ell+1}k_{\ell}^2+\mathcal{R}_{F})\right)\\
&=\sum_{a\in\zb_{p^t}^*}e\left(\frac{a\mathcal{R}_{F}}{p^t}\right)\left(\sum_{r\in\zb_{p^t}}e\left(\frac{ar^2}{p^t}\right)\right)^{\ell-1}\sum_{k_{\ell}\in\zb_{p^t}}e\left(\frac{a\Delta_{F}^{\ell+1}k_{\ell}^2}{p^t}\right).
\end{align*}
holds for all $(p, 2\Delta_{F})=1$.
By Theorem 6 of Chapter 17 in \cite{Hua1957}, we have
\begin{equation}\label{6ss9}
G(p^t, c)=\sum_{r\in\zb_{p^t}}e\left(\frac{cr^2}{p^t}\right)=\begin{cases}\quad p^{\frac{t}{2}}\qquad& if~t~is~even\\
p^{\frac{t}{2}}\left(\frac{c}{p}\right)\chi_{p}\quad\qquad& if~t~is~odd
\end{cases}
\end{equation}
and
\begin{equation*}
\chi_p=\begin{cases}\quad 1\quad \quad& p\equiv 1\bmod 4\\
\sqrt{-1} \quad& p\equiv -1\bmod 4
\end{cases}
\end{equation*}
when $p$ is an odd prime and $c\in\zb$ satisfy $(p,c)=1$. Thus
\begin{align*}
p^{t}S_F(p^t)&=\begin{cases}\sum\limits_{a\in\zb_{p^t}^*}e\left(\frac{a\mathcal{R}_{F}}{p^t}\right)\left(p^{\frac{t}{2}}\right)^{\ell-1}p^{\frac{t}{2}}\quad~~\qquad\qquad\qquad\quad & if~t~is~even\\
\sum\limits_{a\in\zb_{p^t}^*}e\left(\frac{a\mathcal{R}_{F}}{p^t}\right)\left(p^{\frac{t}{2}}{\left(\frac{a}{p}\right)}\chi_{p}\right)^{\ell-1}p^{\frac{t}{2}}\left(\frac{a\Delta_{F}^{\ell+1}}{p}\right)\chi_{p} \quad&  if~t~is~odd
\end{cases}\\
&=\begin{cases}p^{\frac{\ell t}{2}}\sum\limits_{a\in\zb_{p^t}^*}e\left(\frac{a\mathcal{R}_{F}}{p^t}\right)\qquad\quad~~~~\qquad\qquad\qquad\qquad&  if~t~is~even\\
p^{\frac{\ell t}{2}}\chi_{p}^{\ell}\sum\limits_{a\in\zb_{p^t}^*}e\left(\frac{a\mathcal{R}_{F}}{p^t}\right)\left(\frac{a}{p}\right)^{\ell-1}\left(\frac{a\Delta_{F}^{\ell+1}}{p}\right)\qquad\qquad & if~t~is~odd.
\end{cases}
\end{align*}
Recall the properties of Legendre symbol, then
\begin{align}\label{25}
p^{t}S_F(p^t)=\begin{cases}\qquad\qquad p^{\frac{\ell t}{2}}c_{p^t}\left(\mathcal{R}_{F}\right)\qquad\qquad\qquad\qquad\qquad\qquad& if~t~is~even\\
\begin{cases}
p^{\frac{\ell t}{2}}\chi_{p}^{\ell}\sum\limits_{a\in\zb_{p^t}^*}e\left(\frac{a\mathcal{R}_{F}}{p^t}\right)\left(\frac{a}{p}\right)\qquad& if~\ell~is~odd\\
p^{\frac{\ell t}{2}}\chi_{p}^{\ell}\sum\limits_{a\in\zb_{p^t}^*}e\left(\frac{a\mathcal{R}_{F}}{p^t}\right)\left(\frac{\Delta_{F}}{p}\right)& if~\ell~is~even
\end{cases}~ &if~t~is~odd,
\end{cases}
\end{align}
where $c_{p^t}(\mathcal{R}_{F})$ is Ramanujan sum.
It is easily seen that
\begin{equation}\label{26}
\sum\limits_{a\in\zb_{p^t}^*}e\left(\frac{a\mathcal{R}_{F}}{p^t}\right)\left(\frac{a}{p}\right)=
{\bf 1}_{\mathcal{R}_{F}\equiv 0\bmod p^{t-1}}(\mathcal{R}_{F})p^{t-1}\sum\limits_{a\in\zb_{p}^* }\left(\frac{a}{p}\right)e\left(\frac{a\mathcal{R}_{F}}{p^t}\right)
\end{equation}
and
\begin{equation}\label{27}
\sum\limits_{a\in\zb_{p^t}^*}e\left(\frac{a\mathcal{R}_{F}}{p^t}\right)\left(\frac{\Delta_{F}}{p}\right)=
{\bf 1}_{\mathcal{R}_{F}\equiv 0\bmod p^{t-1}}(\mathcal{R}_{F})p^{t-1}\left(\frac{\Delta_{F}}{p}\right)\sum\limits_{a\in\zb_{p}^* }e\left(\frac{a\mathcal{R}_{F}}{p^t}\right).
\end{equation}
If $\mathcal{R}_{F}\equiv 0\bmod p^{t-1}$, then
\begin{align}\label{28}
\sum\limits_{a\in\zb_{p}^* }\left(\frac{a}{p}\right)e\left(\frac{a\mathcal{R}_{F}}{p^t}\right)&=2\sum_{\substack{a\in\zb_{p}^*\\a~is~a~quadratic\\~residue~modulo~p}}
e\left(\frac{a\mathcal{R}_{F}}{p^t}\right)-\sum\limits_{a\in\zb_{p}^*}e\left(\frac{a\mathcal{R}_{F}}{p^t}\right)\nonumber\\
&=\sum_{b=1}^{p-1}\left(
e\left(\frac{b^2\mathcal{R}_{F}}{p^t}\right)-e\left(\frac{b\mathcal{R}_{F}}{p^t}\right)\right)
={\bf 1}_{p^{t}\nmid \mathcal{R}_{F}}\sqrt{p}\left(\frac{\mathcal{R}_{F}/p^{t-1}}{p}\right)\chi_p.
\end{align}
When $\ell$ is an even positive integer, then $\chi_{p}^{\ell}=\left(\frac{(-1)^{\ell/2}}{p}\right)$ for every odd prime number. Note that $\varphi(p^t)=p^{t}(1-p^{-1})$ and
\[c_{n}(m)=\mu(n/(m, n))\varphi(n)/\varphi(n/(m, n)),\]
then by (\ref{25}), (\ref{26}), (\ref{27}) and (\ref{28}) we have the following lemma.
\begin{lemma}\label{tm22}Let $t\ge 1$ be an integer and $p$ satisfies $(p, 2\Delta_{F})=1$, we have

1). If  $\mathcal{R}_{F}\equiv 0\bmod p^{t}$, then
\begin{align*}
S_F(p^t)=\begin{cases}\qquad\qquad (1-p^{-1})p^{\frac{\ell t}{2}}\qquad\qquad\qquad\qquad\qquad\qquad& if~ t~ is~ even\\
\begin{cases}
\qquad\qquad\quad 0\qquad\qquad\qquad\qquad& if~ \ell~ is~ odd\\
(1-p^{-1})p^{\frac{\ell t}{2}}\left(\frac{(-1)^{{\ell}/{2}}\Delta_{F}}{p}\right)\qquad\qquad& if~ \ell~ is~ even
\end{cases}& if~t~is~odd.
\end{cases}
\end{align*}

2). If $\mathcal{R}_{F}\equiv 0\bmod p^{t-1}$ but $\mathcal{R}_{F}\not\equiv 0\bmod p^t$, then
\begin{align*}
S_F(p^t)=\begin{cases}\qquad\qquad -p^{\frac{\ell t-2}{2}}\qquad\qquad\qquad\qquad\qquad\qquad\qquad\quad& if~ t~ is~ even\\
\begin{cases}
p^{\frac{\ell t-1}{2}}\left(\frac{(-1)^{\frac{\ell+1}{2}}\mathcal{R}_{F}/p^{t-1}}{p}\right)\qquad& if~ \ell~ is~ odd\\
-p^{\frac{\ell t-2}{2}}\left(\frac{(-1)^{\ell/2}\Delta_{F}}{p}\right)\qquad\qquad& if~ \ell~ is~ even
\end{cases}\qquad& if~t~is~odd.
\end{cases}
\end{align*}

3). If $\mathcal{R}_{F}\not\equiv 0\bmod p^{t-1}$, then
$
S_F(p^t)=0
$.
\end{lemma}
\begin{corollary}\label{tm23}
Let $t\ge 1$ be an integer and prime number $p$ satisfies $(p, 2\Delta_{F}\mathcal{R}_{F})=1$. Then
\begin{align*}
S_F(p^t)=\begin{cases}
\begin{cases}
p^{\frac{\ell -1}{2}}\left(\frac{(-1)^{\frac{\ell+1}{2}}\mathcal{R}_{F}}{p}\right)\qquad\qquad& if~ \ell~ is~ odd\\
-p^{\frac{\ell -2}{2}}\left(\frac{(-1)^{\ell/2}\Delta_{F}}{p}\right)\qquad\qquad& if~ \ell~ is~ even
\end{cases}& if~t=1\\
\qquad\qquad\qquad 0\qquad\qquad\qquad\qquad\qquad\quad& if~t\ge 2.
\end{cases}
\end{align*}
\end{corollary}
\subsection{The treatment of $L(s,F)$} Firstly, we consider $\mathcal{R}_{F}=0$.
In this case, we have firstly
\begin{equation}\label{6311}
L(s, F)=\prod_{p|2\Delta_{F}}\left(1+\sum_{m\ge 1}S_F(p^m)p^{-ms}\right)\prod_{p\nmid 2\Delta_{F}}\left(1+\sum_{m\ge 1}S_F(p^m)p^{-ms}\right).
\end{equation}
If $\ell\ge 3$ is an odd integer then by Lemma \ref{tm22} one has $S_F(p^{2m+1})=0$. Therefore the above second product equals to
\begin{align*}
\prod_{p\nmid 2\Delta_{F}}\sum_{m\ge 0}\frac{S_F(p^{2m})}{p^{2ms}}&=\prod_{p\nmid 2\Delta_{F}}\left(1+(1-p^{-1})\sum_{m\ge 1}p^{\frac{\ell 2m}{2}}p^{-2ms}\right)\\
&=\prod_{p|2\Delta_{F}}\left(1+(1-p^{-1})\sum_{m\ge 1}p^{(\ell-2s)m}\right)^{-1}\prod_{p}\left(1+(1-p^{-1})\sum_{m\ge 1}p^{(\ell-2s)m}\right)\\
&=\prod_{p|2\Delta_{F}}\left(\frac{1-p^{\ell-2s-1}}{1-p^{\ell-2s}}\right)^{-1}\prod_{p}\left(\frac{1-p^{\ell-2s-1}}{1-p^{\ell-2s}}\right)\\
&=\frac{\zeta(2s-\ell)}{\zeta(2s+1-\ell)}\prod_{p|2\Delta_{F}}\left(\frac{1-p^{\ell-2s}}{1-p^{\ell-2s-1}}\right).
\end{align*}
Hence we have
\begin{align}\label{211}
L(s, F)=f(s, F){\zeta(2s-\ell)}{\zeta(2s+1-\ell)}^{-1}
\end{align}
and where
\begin{align}\label{ssff}
f(s, F)=\prod_{p|2\Delta_{F}}\left\{\left(\frac{1-p^{\ell-2s}}{1-p^{\ell-2s-1}}\right)\left(1+\sum_{m\ge 1}S_F(p^m)p^{-ms}\right)\right\}.
\end{align}
On the other hand, by (\ref{21}) we obtain that
\begin{align}\label{ssff1}
\prod_{p|2\Delta_{F}}\left(1+\sum_{m\ge 1}S_F(p^m)p^{-ms}\right)=
\prod_{p|2\Delta_{F}}\left\{\left(1-p^{\ell-s-1}\right)\sum_{m\ge 0}\frac{\varrho_{F}(p^m)}{p^{ms}}\right\}.
\end{align}
Hence
\begin{align}\label{ssff2}
f(s, F)=\prod_{p|2\Delta_{F}}\left(\frac{(1-p^{\ell-2s})(1-p^{\ell-s-1})}{1-p^{\ell-2s-1}}\sum_{m\ge 0}\frac{\varrho_{F}(p^m)}{p^{ms}}\right).
\end{align}
The next we consider the case $\ell$ is even. Similarly we have the second product factor of (\ref{6311}) equals to
\begin{align*}
&=1+(1-p^{-1})\sum_{m\ge 1}\left(p^{-(2m-1)s}p^{\frac{\ell (2m-1)}{2}}\left(\frac{(-1)^{\frac{\ell}{2}}\Delta_{F}}{p}\right)+p^{-2ms}p^{\ell m}\right)\\
&=1+(1-p^{-1})\left(1+p^{s-\frac{\ell}{2}}\left(\frac{(-1)^{\frac{\ell}{2}}\Delta_{F}}{p}\right)\right)\sum_{m\ge 1}p^{(\ell-2s)m}\\
&=1+(1-p^{-1})\left(1+p^{s-\frac{\ell}{2}}\left(\frac{(-1)^{\frac{\ell}{2}}\Delta_{F}}{p}\right)\right)\left(\frac{p^{\ell-2s}}{1-p^{\ell-2s}}\right)\\
&=\left(\frac{1-p^{\ell-2s-1}}{1-p^{\ell-2s}}\right)\left(1+\left(\frac{(-1)^{\frac{\ell}{2}}\Delta_{F}}{p}\right)\frac{(1-p^{-1})p^{\frac{\ell}{2}-s}}{1-p^{\ell-2s-1}}\right).
\end{align*}
On the other hand, if $p|\Delta_{F}$ then one has $\left(\frac{\Delta_{F}}{p}\right)=0$. Hence similar with the case $\ell$ is odd, we obtain that
\begin{align}\label{216}
L(s, F)=f(s, F)\frac{\zeta(2s-\ell)}{\zeta(2s+1-\ell)}
\prod_{p>2}\left(1+\left(\frac{(-1)^{\frac{\ell}{2}}\Delta_{F}}{p}\right)\frac{(1-p^{-1})p^{\frac{\ell}{2}-s}}{1-p^{\ell-2s-1}}\right),
\end{align}
where $f(s, F)$ is defined same as (\ref{ssff}) and (\ref{ssff2}).
Using the definition notations (\ref{def12}) and (\ref{def13}). Then, we have
\begin{equation}\label{ll1}
L(s, F)=f(s, F)\frac{\zeta(2s-\ell)}{\zeta(2s+1-\ell)}\prod\limits_{p>2}\left(1+\left(\frac{\mathcal{H}_{F}}{p}\right)
\frac{(1-p^{-1})p^{\frac{\ell}{2}-s}}{1-p^{\ell-2s-1}}\right),
\end{equation}
where
\[f(s, F)=\prod_{p|2\Delta_{F}}\left(\frac{(1-p^{\ell-2s})(1-p^{\ell-s-1})}{1-p^{\ell-2s-1}}\sum_{m\ge 0}\frac{\varrho_{F}(p^m)}{p^{ms}}\right).\]
\newline

Let us consider the cases $\mathcal{R}_{F}\neq 0$. By the above Corollary \ref{tm23}, if $(p, 2\Delta_{F}\mathcal{R}_{F})=1$ then one has
\begin{align*}
S_F(p)=
\begin{cases}
p^{\frac{\ell -1}{2}}\left(\frac{(-1)^{\frac{\ell+1}{2}}\mathcal{R}_{F}}{p}\right)\qquad\qquad& if~ \ell~ is~ odd\\
-p^{\frac{\ell -2}{2}}\left(\frac{(-1)^{\ell/2}\Delta_{F}}{p}\right)\qquad\qquad& if~ \ell~ is~ even.
\end{cases}
\end{align*}
and if $t\ge 2$ then $S(F;p^t)=0$. Therefore we get that
$$
\prod_{p\nmid 2\Delta_{F}\mathcal{R}_{F}}\left(1+\sum_{m\ge 1}S_F(p^m)p^{-ms}\right)=\prod_{p\nmid 2\Delta_{F}\mathcal{R}_{F}}\left(1+S_F(p)p^{-s}\right).
$$
Hence
\begin{align*}
\prod_{p\nmid 2\Delta_{F}\mathcal{R}_{F}}\left(1+S_F(p)p^{-s}\right)=\begin{cases}\prod\limits_{p\nmid 2\Delta_{F}\mathcal{R}_{F}}\left(1+\left(\frac{(-1)^{\frac{\ell+1}{2}}\mathcal{R}_{F}}{p}\right)p^{-(s-\frac{\ell-1}{2})}\right)\quad& if ~\ell~ is~ odd\\
\prod\limits_{p\nmid 2\Delta_{F}\mathcal{R}_{F}}\left(1-\left(\frac{(-1)^{\ell/2}\Delta_{F}}{p}\right)p^{-(s-\frac{\ell-2}{2})}\right)\qquad& if~ \ell~ is~ even.
\end{cases}
\end{align*}
Similar with (\ref{211}) or (\ref{216}), use the definition notations (\ref{def12}) and (\ref{def13}). Then, we have
\begin{equation}\label{ll2}
L(s, F)=h(s, F)\prod\limits_{p\nmid 2\Delta_{F}\mathcal{R}_{F}}\left\{\left(1+\left(\frac{\mathcal{O}_{F}}{p}\right)p^{-(s-\frac{\ell-1}{2})}\right)\left(1-\left(\frac{\mathcal{H}_{F}}{p}\right)p^{-(s-\frac{\ell-2}{2})}\right)\right\},
\end{equation}
where
$$h(s, F)=\prod_{p|2\Delta_{F}\mathcal{R}_{F}}\left\{\left(1-p^{\ell-s-1}\right)\sum_{m\ge 0}\frac{\varrho_{F}(p^m)}{p^{ms}}\right\}.$$
Then combine (\ref{ll0}), (\ref{ll1}) and (\ref{ll2}), we get the proof of the theorem.
\section{Examples}
In the section, we give the computation of some example. We always assume that $\ell\in\zb_{\ge 1}$.
\begin{lemma}\label{tm31}For all $(b, 2)=1$, we have
\begin{align*}
\sum_{r\in\mathbb{Z}_{2^t}}e\left(\frac{br^2}{2^t}\right)=\begin{cases}\qquad 0\qquad\qquad\qquad
& t=1 \\
\quad 2^{\frac{t+1}{2}}e(\frac{b}{8})\qquad\qquad& t>1, t\equiv 1\bmod 2 \\
\left(1+(\sqrt{-1})^{b}\right)2^{\frac{t}{2}}\qquad& t\equiv
0\bmod 2
\end{cases}.
\end{align*}
\end{lemma}
\begin{proof}
The lemma is the Theorem 8 of Chapter 17 in \cite{Hua1957}.
\end{proof}
If $t=1$, then we have
\begin{equation}\label{72}
\sum_{1\le b\le 2, 2\nmid b}\left(\sum_{r\in\mathbb{Z}_{2}}e\left(\frac{br^2}{2}\right)\right)^{\ell}=0.
\end{equation}
If $t>1$ is an odd integer, then
\begin{align}\label{73}
\sum_{1\le b\le 2^t, 2\nmid b}\left(\sum_{r\in\mathbb{Z}_{2^t}}e\left(\frac{br^2}{2^t}\right)\right)^{\ell}&=
2^{\ell\frac{t+1}{2}}\sum_{1\le b\le 2^{t-1}}e\left(\frac{\ell (2b-1)}{8}\right)=
2^{\ell\frac{t+1}{2}}e\left(\frac{-\ell}{8}\right)\sum_{1\le b\le 2^{t-1}}e\left(\frac{\ell b}{4}\right)\nonumber\\
&=
\begin{cases}
\qquad\quad 0\qquad\qquad\qquad& 4\nmid \ell, \\
2^{\frac{\ell+2}{2}t+\frac{\ell}{2}-1}(-1)^{\ell/4}\quad&\ell\equiv 0\bmod 4.
\end{cases}
\end{align}
If $t>1$ is an even integer, then
\begin{align}\label{74}
\sum_{1\le b\le 2^t, 2\nmid b}\left(\sum_{r\in\mathbb{Z}_{2^t}}e\left(\frac{br^2}{2^t}\right)\right)^{\ell}&=
2^{\ell\frac{t}{2}}\sum_{b=1}^{2^{t-1}}\left(1+(\sqrt{-1})^{2b-1}\right)^{\ell}=
2^{\ell\frac{t}{2}}\sum_{b=1}^{2^{t-1}}\left(1+(-1)^{b}\sqrt{-1}\right)^{\ell}\nonumber\\
&=2^{\ell\frac{t}{2}}\sum_{b=1}^{2^{t-2}}\left(\left(1-\sqrt{-1}\right)^{\ell}+\left(1+\sqrt{-1}\right)^{\ell}\right)
=2^{\frac{\ell+2}{2}t+\frac{\ell}{2}-1}\cos{\frac{\pi\ell}{4}}.
\end{align}
Combining (\ref{72}-\ref{74}) and Lemma \ref{tm31} one has the lemma.
\begin{lemma}
\begin{align}\label{75}
2^{-t}\sum_{1\le b\le 2^t, 2\nmid b}\left(\sum_{r\in\mathbb{F}_{2^t}}e\left(\frac{br^2}{2^t}\right)\right)^{\ell}
=\begin{cases}
\quad\quad\quad~~ 0\quad\quad\quad& t=1~ or~ t~ is~ odd~ and~ 4\nmid \ell, \\
2^{\frac{\ell}{2}t+\frac{\ell}{2}-1}\cos{\frac{\pi\ell}{4}}& t>1, ~4|\ell ~or ~t>1~ is ~even~ and ~4\nmid \ell.
\end{cases}
\end{align}
\end{lemma}

Now, we give the following examples
\begin{example}Let quadratic form
$$F=k_1^2+k_2^2+. . . +k_{\ell}^{2}$$
Then by. we have $\Delta_{F}=2^{\ell}$, and
\begin{align*}
\sum_{m\ge 1}\frac{S_F(2^m)}{2^{ms}}&=\sum_{m\ge 1}\left(S_F(2^{2m})2^{-2ms}+S_F(2^{2m+1})2^{-({2m+1})s}\right)\\
&=\begin{cases}
\sum_{m\ge 1}S_F(2^{2m})2^{-2ms}\quad  &4\nmid \ell\\
\sum_{m\ge 2}S_F(2^{m})2^{-ms}\quad  &4\mid \ell
\end{cases}=\begin{cases}
\sum_{m\ge 1}2^{(\ell-2s)m+\frac{\ell}{2}-1}\cos{\frac{\pi\ell}{4}}\quad  &4\nmid \ell\\
\sum_{m\ge 2}2^{(\frac{\ell}{2}-s)m+\frac{\ell}{2}-1}\cos{\frac{\pi\ell}{4}}\quad  &4\mid \ell
\end{cases}\\
&=2^{\frac{3\ell}{2}-2s-1}\cos{\frac{\pi\ell}{4}}\times\begin{cases}({1-2^{\ell-2s}})^{-1}\quad  &4\nmid \ell\\
({1-2^{\ell/2-s}})^{-1}\quad  &4\mid \ell.
\end{cases}
\end{align*}
If $\ell\equiv 0\bmod 4$, then $\left(\frac{\mathcal{H}_{F}}{p}\right)=1$ and
\begin{align*}
L(s,F)&=\left(\frac{1-2^{\ell-2s}}{1-2^{\ell-2s-1}}\right)\left(1+\frac{2^{\frac{3\ell}{2}-2s-1}(-1)^{\ell/4}}{1-2^{\ell/2-s}}\right)
\frac{\zeta(2s-\ell)}{\zeta(2s+1-\ell)}\prod_{p>2}\left(1+\frac{(1-p^{-1})p^{\frac{\ell}{2}-s}}{1-p^{\ell-2s-1}}\right)\\
&=\left(\frac{1-2^{\ell-2s}}{1-2^{\ell-2s-1}}\right)\left(1+\frac{2^{\frac{3\ell}{2}-2s-1}(-1)^{\frac{\ell}{4}}}{1-2^{\frac{\ell}{2}-s}}\right)
\frac{\zeta(2s-\ell)}{\zeta(2s+1-\ell)}\prod_{p>2}\frac{(1-p^{\frac{\ell}{2}-s-1})(1-p^{{\ell}-2s})}{(1-p^{\ell-2s-1})(1-p^{\frac{\ell}{2}-s})}\\
&=\left(1+\frac{2^{{3\ell}/{2}-2s-1}(-1)^{\ell/4}}{1-2^{\ell/2-s}}\right)\left(\frac{1-2^{{\ell}/{2}-s-1}}{1-2^{{\ell}/{2}-s}}\right)^{-1}
\prod_{p}\frac{1-p^{{\ell}/{2}-s-1}}{1-p^{{\ell}/{2}-s}}\\
&=\frac{1-2^{{\ell}/{2}-s}+(-1)^{\ell/4}2^{{3\ell}/{2}-2s-1}}{1-2^{{\ell}/{2}-s-1}}\frac{\zeta(s-{\ell}/{2})}{\zeta(s+1-{\ell}/{2})}.
\end{align*}
Thus we obtain that
\[L(\ell,F)=\frac{2^{\frac{\ell}{2}+1}-2+(-1)^{\ell/4}}{2^{\frac{\ell}{2}+1}-1}\frac{\zeta(\frac{\ell}{2})}{\zeta(\frac{\ell}{2}+1)}\]
and
\[\frac{L'(\ell,F)}{L(\ell,F)}=\frac{\zeta '\left(\frac{\ell}{2}\right)}{\zeta \left(\frac{\ell}{2}\right)}-\frac{\zeta '\left(\frac{\ell}{2}+1\right)}{\zeta \left(\frac{\ell}{2}+1\right)}+\frac{(-1)^{\ell/4}+2^{\frac{\ell}{2}+1}-(-1)^{\frac{\ell}{4}} 2^{\frac{\ell}{2}+2}}{\left(2^{\frac{\ell}{2}+1}-1\right)\left(2^{\frac{\ell}{2}+1}-2+(-1)^{\ell/4}\right)}\log 2.
\]

If $\ell\equiv \pm 1\bmod 4$, then $\left(\frac{\mathcal{H}_{F}}{p}\right)=0$ and
\begin{align*}
L(s,Q_F)&=\left(\frac{1-2^{\ell-2s}}{1-2^{\ell-2s-1}}\right)\left(1+\frac{2^{\frac{3\ell}{2}-2s-1}\cos{\frac{\pi\ell}{4}}}{{1-2^{\ell-2s}}}\right)\frac{\zeta(2s-\ell)}{\zeta(2s+1-\ell)}\\
&=\frac{1-2^{\ell-2s}+2^{\frac{3\ell}{2}-2s-1}\cos{\frac{\pi\ell}{4}}}{{1-2^{\ell-2s-1}}}\frac{\zeta(2s-\ell)}{\zeta(2s+1-\ell)}
\end{align*}
Hence
\[
L(\ell,F)=\left(1-\frac{1-2^{\frac{\ell}{2}}\cos{\frac{\pi\ell}{4}}}{{2^{\ell+1}-1}}\right)\frac{\zeta(\ell)}{\zeta(\ell+1)}
\]
and
\[
\frac{L'(\ell,F)}{L(\ell,F)}=\frac{\zeta '\left({\ell}\right)}{\zeta \left({\ell}\right)}-\frac{\zeta '\left({\ell}+1\right)}{\zeta \left({\ell}+1\right)}
+\frac{2^{\ell+2} \left(1-2^{\ell/2} \cos \left(\frac{\ell\pi}{4}\right)\right)\log 2}{\left(2^{\ell+1}-1\right)\left(2^{\ell+1}-2+2^{\frac{\ell}{2}}\cos{\frac{\pi\ell}{4}}\right)}.
\]

If $\ell\equiv 2\bmod 4$, then $\left(\frac{\mathcal{H}_{F}}{p}\right)=\left(\frac{-1}{p}\right)$ and
\begin{align*}
L(s,F)&=\left(\frac{1-2^{\ell-2s}}{1-2^{\ell-2s-1}}\right)\left(1+\frac{2^{\frac{3\ell}{2}-2s-1}\cos{\frac{\pi\ell}{4}}}{{1-2^{\ell-2s}}}\right)\frac{\zeta(2s-\ell)}{\zeta(2s+1-\ell)}\prod_{p>2}\left(1+\left(\frac{-1}{p}\right)\frac{(1-p^{-1})p^{\frac{\ell}{2}-s}}{1-p^{\ell-2s-1}}\right)\\
&=\left(\frac{1-2^{\ell-2s}}{1-2^{\ell-2s-1}}\right)\frac{\zeta(2s-\ell)}{\zeta(2s+1-\ell)}\prod_{p>2}\left(1+\left(\frac{-1}{p}\right)\frac{(1-p^{-1})p^{\frac{\ell}{2}-s}}{1-p^{\ell-2s-1}}\right)\\
&=\prod_{p>2}\frac{1-p^{\ell-2s-1}+\left(\frac{-1}{p}\right)(1-p^{-1})p^{\frac{\ell}{2}-s}}{1-p^{\ell-2s}}=\prod_{p>2}\frac{\left(1+\left(\frac{-1}{p}\right)p^{\frac{\ell}{2}-s}\right)\left(1-\left(\frac{-1}{p}\right)p^{\frac{\ell}{2}-s-1}\right)}{1-p^{\ell-2s}}\\
&=\prod_{p>2}\frac{1-\left(\frac{-1}{p}\right)p^{\frac{\ell}{2}-s-1}}{1-\left(\frac{-1}{p}\right)p^{\frac{\ell}{2}-s}}=\frac{L(s-\ell/2,\chi)}{L(s+1-\ell/2,\chi)},
\end{align*}
where $\chi$ is the non-principal Dirichlet character modulo $4$. Hence, we have
\[
L(\ell,F)=\frac{L(\ell/2,\chi)}{L(\ell/2+1,\chi)}\;\mbox{and}\; \frac{L'(\ell,F)}{L(\ell,F)}=\frac{L'(\ell/2,\chi)}{L(\ell/2,\chi)}-\frac{L'(\ell/2+1,\chi)}{L(\ell/2+1,\chi)}.
\]
\end{example}

\begin{example}Let quadratic form be
\begin{align*}
Q_{FT}=k_1k_2+k_3k_4+. . . +k_{2\ell-1}k_{2\ell}~~\text{and}~~
Q_{FM}=k_0^2+k_1k_2+. . . +k_{2\ell-1}k_{2\ell}.
\end{align*}
$$\Delta_{Q_{FT}}=(-1)^{\ell},~\Delta_{Q_{FM}}=2(-1)^{\ell},~\mathcal{H}_{Q_{FT}}=1,\quad\quad \mathcal{H}_{Q_{FM}}=0,$$
$$S(Q_{FT},2^m)=2^{-m}\sum_{b\in\zb_{2^m}^*}\left(\sum_{h_1,h_2\in\zb_{2^m}}e\left(\frac{bh_1h_2}{2^m}\right)\right)^{\ell}=2^{m\ell-1},$$
\begin{align*}
S(Q_{FM}, 2^m)&=2^{-m}\sum_{b\in\zb_{2^m}^*}\sum_{h_0\in\zb_{2^m}}e\left(\frac{bh_0^2}{2^m}\right)\left(\sum_{h_1, h_2\in\zb_{2^m}}e\left(\frac{bh_1h_2}{2^m}\right)\right)^{\ell}\\
&=2^{-m}\sum_{b\in\zb_{2^m}^*}\sum_{h_0\in\zb_{2^m}}e\left(\frac{bh_0^2}{2^m}\right)2^{m\ell}
=\begin{cases}\quad\quad 0\qquad m~is~odd\\
2^{\frac{(2\ell+1)m}{2}-1}\quad~ m~is~even.
\end{cases}
\end{align*}
Hence
\begin{align*}
f(s, Q_{FT})&=\left(\frac{1-2^{2\ell-2s}}{1-2^{2\ell-2s-1}}\right)\left(1+\sum_{m\ge 1}\frac{S(Q_{FM}, 2^m)}{2^{ms}}\right)\\
&=\left(\frac{1-2^{2\ell-2s}}{1-2^{2\ell-2s-1}}\right)\left(\frac{1-2^{\ell-1-s}}{1-2^{\ell-s}}\right)=1+\frac{2^{\ell-s-1}}{1-2^{2\ell-1-2s}},
\end{align*}
\begin{align*}
f(s, Q_{FM})&=\left(\frac{1-2^{2\ell+1-2s}}{1-2^{2\ell-2s}}\right)\left(1+\sum_{m\ge 1}\frac{S(Q_{FT}, 2^m)}{2^{ms}}\right)\\
&=\left(\frac{1-2^{2\ell+1-2s}}{1-2^{2\ell-2s}}\right)\left(\frac{1-2^{2\ell-2s}}{1-2^{2\ell+1-2s}}\right)=1,
\end{align*}
\begin{align*}
L(s, Q_{FT})=\frac{\zeta(s-\ell)}{\zeta(s+1-\ell)},~L(s, Q_{FM})=\frac{\zeta(2s-2\ell-1)}{\zeta(2s-2\ell)}.
\end{align*}
We have
$$
\sum_{1\le x_1, x_2, x_3, x_4\le X}\tau\left(x_1x_2+x_3x_4\right)=\left(c_{\mathcal{TT}}(\tau)_mX^{4}\log X+c_{\mathcal{TT}}(\tau)_sX^4\right)+O\left( X^{4-\frac{1}{3}+\varepsilon}\right),
$$
where
$$
c_{\mathcal{TT}}(\tau)_m=\frac{2\zeta(2)}{\zeta(3)} \quad \text{and}\quad
c_{\mathcal{TT}}(\tau)_s=\frac{2\zeta(2)}{\zeta(3)}\left(\gamma+\frac{\pi ^2}{48}-\frac{11}{6}+\log 2+\frac{\zeta'(2)}{\zeta(2)}-\frac{\zeta'(3)}{\zeta(3)}\right).
$$
We have
$$
\sum_{1\le x_0, x_1, x_2\le X}\tau\left(x_0^2+x_1x_2\right)=\left(c_{\mathcal{OT}}(\tau)_mX^{3}\log X+c_{\mathcal{OT}}(\tau)_sX^3\right)+O\left( X^{3-\frac{1}{5}+\varepsilon}\right),
$$
where
$$
c_{\mathcal{OT}}(\tau)_m=\frac{2\zeta(3)}{\zeta(4)} ~~\text{and}~~
c_{\mathcal{OT}}(\tau)_s=\frac{4\zeta(3)}{\zeta(4)}\left(\frac{\gamma}{2}+\frac{\pi ^2+8 \pi -104+56 \log 2}{144}+\frac{\zeta'(3)}{\zeta(3)}-\frac{\zeta'(4)}{\zeta(4)}\right).
$$
\end{example}
\begin{example}
Consider the following quadratic form,
$$Q_{sym}=xy+yz+xz.$$

$$\mathcal{H}_{Q_{sym}}=0$$,
$$
L(s, Q_{FM})=\zeta(2s-3)\zeta(2s-2)^{-1}.
$$
We have
$$
\sum_{1\le x, y, z\le X}\tau\left(xy+yz+xz\right)=\left(c_{{sym}}(\tau)_mX^{3}\log X+c_{{sym}}(\tau)_sX^3\right)+ O\left(X^{3-\frac{1}{5}+\varepsilon}\right),
$$
where
$$
c_{{sym}}(\tau)_m=\frac{2\zeta(2)}{\zeta(3)}~~\text{and}~~
c_{{sym}}(\tau)_s=\frac{4\zeta(2)}{\zeta(3)}\left(\frac{\gamma}{2}+\frac{1}{4}\int_{[0,1]^3}\mathrm{d}{\bf t}\log(t_1t_2+t_2t_3+t_1t_3)+\frac{\zeta'(2)}{\zeta(2)}-\frac{\zeta'(3)}{\zeta(3)}\right).
$$
\end{example}

\begin{example}
Consider the following quadratic,
$$F=x^2+y^2+z^2+1$$
its
$$\Delta_{F}=2^3, \quad \mathcal{R}_{F}=2^4$$
Similar with (\ref{75}), we have
\begin{align}\label{78}
S(F, 2^m)=2^{-m}\sum_{\substack{1\le b\le 2^{m}\\(2, b)=1}}e\left(\frac{b}{2^m}\right)\left(\sum_{x\bmod 2^m}e\left(\frac{bx^2}{2^m}\right)\right)^3=0\qquad m\ge 1.
\end{align}
Then
\[h(s, F)=1.
\]
Hence we have
\begin{align}
L(s, F)=\prod\limits_{p>2}\left(1+p^{-s+1}\right)=\frac{1}{1+2^{1-s}}\frac{\zeta(s-1)}{\zeta(2s-2)}.
\end{align}
We have
$$\sum_{1\le x, y, z\le X}\tau(x^2+y^2+z^2+1)=\left(c_{\mathcal{D}_l}(\tau)_mX^{3}\log X+c_{\mathcal{D}_l}(\tau)_sX^3\right)+O\left(X^{\ell-\frac{1}{5}+\varepsilon}\right),$$
where
$$
c_{\mathcal{D}_l}(\tau)_m=\frac{24}{\pi^2}~~\text{and}~~
c_{\mathcal{D}_l}(\tau)_s=\frac{24}{\pi^2}\left(\gamma+\frac{1}{2}\int_{[0,1]^3}\mathrm{d}{\bf t}\log(t_1^2+t_2^2+t_3^2)+\frac{\log 2}{5}+\frac{\zeta'(2)}{\zeta(2)}-2\frac{\zeta'(4)}{\zeta(4)}\right).
$$
\end{example}

\end{document}